\documentclass[11pt,reqno]{amsart} %set font size up here (in [])
\usepackage{amssymb,amsthm,amsmath,epsfig,latexsym}
\usepackage{graphicx}
\usepackage{calc,times}
\usepackage{tikz}
\usetikzlibrary{matrix}
\usepackage[all]{xy}

\usepackage{geometry} %geometry package allows easy manipulation of margins, etc.
\begin{document}

% \mmbox enables macros to survive outside of $ ... $
\newcommand{\mmbox}[1]{\mbox{${#1}$}}
\newcommand{\proj}[1]{\mmbox{{\mathbb P}^{#1}}}
\newcommand{\Cr}{C^r(\Delta)}
\newcommand{\CR}{C^r(\hat\Delta)}
\newcommand{\affine}[1]{\mmbox{{\mathbb A}^{#1}}}
\newcommand{\Ann}[1]{\mmbox{{\rm Ann}({#1})}}
\newcommand{\caps}[3]{\mmbox{{#1}_{#2} \cap \ldots \cap {#1}_{#3}}}
\newcommand{\N}{{\mathbb N}}
\newcommand{\Z}{{\mathbb Z}}
\newcommand{\R}{{\mathbb R}}
\newcommand{\Tor}{\mathop{\rm Tor}\nolimits}
\newcommand{\Ext}{\mathop{\rm Ext}\nolimits}
\newcommand{\Hom}{\mathop{\rm Hom}\nolimits}
\newcommand{\im}{\mathop{\rm Im}\nolimits}
\newcommand{\rank}{\mathop{\rm rank}\nolimits}
\newcommand{\supp}{\mathop{\rm supp}\nolimits}
\newcommand{\arrow}[1]{\stackrel{#1}{\longrightarrow}}
\newcommand{\CB}{Cayley-Bacharach}
\newcommand{\coker}{\mathop{\rm coker}\nolimits}
\newcommand{\m}{{\frak m}}
%%%%%%%%%%%%%%%%%%%%%%%%%%%%%%%%%%%%%%%%%%%%%%%%%%%%%%%%%%%%%%%%%%%%%%%%% new commands - Mehdi
\newcommand{\fitt}{{\rm Fitt}}
\newcommand{\C}{{\mathcal{C}}}
\newcommand{\K}{\mathbb K}
\newcommand{\OT}{R}  % Orlik-Terao ring (R = std notation)
\newcommand{\pan}{{\rm Span }}  % I know! \span is taken
\newcommand{\M}{\mathsf M}
\newcommand{\Ima}{{\rm Im}\,}

%%%%%%%%%%%%%%%%%%%%%%%%%%%%%%%%%%%%%%%%%%%%%%%%%%%%%%%%%%%%%%%%%%%%%%%%% for augmented matrices

\makeatletter
\renewcommand*\env@matrix[1][*\c@MaxMatrixCols c]{%
  \hskip -\arraycolsep
  \let\@ifnextchar\new@ifnextchar
  \array{#1}}
\makeatother

%%%%%%%%%%%%%%%%%%%%%%%%%%%%%%%%%%%%%%%%%%%%%%%%%%%%%%

%\DeclareMathOperator{\M}{\mathsf M}
%\newcommand{\text}[1]{\mbox{\rm {#1}}}
\sloppy
\newtheorem{defn0}{Definition}[section]
\newtheorem{prop0}[defn0]{Proposition}
\newtheorem{quest0}[defn0]{Question}
\newtheorem{thm0}[defn0]{Theorem}
\newtheorem{lem0}[defn0]{Lemma}
\newtheorem{corollary0}[defn0]{Corollary}
\newtheorem{example0}[defn0]{Example}
\newtheorem{remark0}[defn0]{Remark}
\newtheorem{prob0}[defn0]{Problem}

\newenvironment{defn}{\begin{defn0}}{\end{defn0}}
\newenvironment{prop}{\begin{prop0}}{\end{prop0}}
\newenvironment{quest}{\begin{quest0}}{\end{quest0}}
\newenvironment{thm}{\begin{thm0}}{\end{thm0}}
\newenvironment{lem}{\begin{lem0}}{\end{lem0}}
\newenvironment{cor}{\begin{corollary0}}{\end{corollary0}}
\newenvironment{exm}{\begin{example0}\rm}{\end{example0}}
\newenvironment{rem}{\begin{remark0}\rm}{\end{remark0}}
\newenvironment{prob}{\begin{prob0}\rm}{\end{prob0}}

\newcommand{\defref}[1]{Definition~\ref{#1}}
\newcommand{\propref}[1]{Proposition~\ref{#1}}
\newcommand{\thmref}[1]{Theorem~\ref{#1}}
\newcommand{\lemref}[1]{Lemma~\ref{#1}}
\newcommand{\corref}[1]{Corollary~\ref{#1}}
\newcommand{\exref}[1]{Example~\ref{#1}}
\newcommand{\secref}[1]{Section~\ref{#1}}
\newcommand{\remref}[1]{Remark~\ref{#1}}
\newcommand{\questref}[1]{Question~\ref{#1}}
\newcommand{\probref}[1]{Problem~\ref{#1}}

\newcommand{\std}{Gr\"{o}bner}
\newcommand{\jq}{J_{Q}}
\def\Ree#1{{\mathcal R}(#1)}

\numberwithin{equation}{subsection}  %% equation numbering style

%\parskip = 4pt

%\begin{singlespace}
\title{Minimum distance of linear codes and the $\alpha$-invariant}
\author{Mehdi Garrousian and \c{S}tefan O. Toh\v{a}neanu}

 \keywords{minimum distance, Fitting ideal, filtration, inverse systems, Orlik-Terao algebra. \\ \indent Garrousian's Address: Departamento de Matem\'aticas, Universidad de los Andes, Cra 1 No. 18A-12, Bogot\'a, Colombia, Email: m.garrousian@uniandes.edu.co\\ \indent Tohaneanu's Address: Department of Mathematics, University of Idaho, Moscow, Idaho 83844-1103, USA, Email: tohaneanu@uidaho.edu, Phone: 208-885-6234, Fax: 208-885-5843.}

\begin{abstract} The simple interpretation of the minimum distance of a linear code obtained by De Boer and Pellikaan, and later refined by the second author, is further developed through the study of various finitely generated graded modules. We use the methods of commutative/homological algebra to find connections between the minimum distance and the $\alpha$-invariant of such modules.

\end{abstract}
\maketitle

\section{Introduction}

Let $\mathcal C$ be an $[n,k,d]$-linear code with generating matrix (in canonical bases) $$G=\left[\begin{array}{cccc}a_{11}&a_{12}&\cdots&a_{1n}\\ a_{21}&a_{22}&\cdots&a_{2n}\\\vdots&\vdots& &\vdots\\ a_{k1}&a_{k2}&\cdots&a_{kn}\end{array}\right],$$ where $a_{ij}\in\mathbb K$, any field.

By this, one understands that $\mathcal C$ is the image of the injective linear map $$\phi:\mathbb K^k\stackrel{G}\longrightarrow \mathbb K^n.$$ $n$ is the {\em length} of $\mathcal C$, $k$ is the {\em dimension} of $\mathcal C$ and $d$ is the {\em minimum distance (or Hamming distance)}, the smallest number of non-zero entries in a non-zero codeword (i.e.\ non-zero element of $\mathcal C$).

Also, for any vector $w\in\mathbb K^n$, the {\em weight} of $w$, denoted $wt(w)$, is the number of non-zero entries in $w$. A vector has at most $m$ non-zero entries if and only if all products of $m+1$ distinct entries are zero. This simple observation was first exploited in the context of coding theory by De Boer and Pellikaan \cite{DePe}, in the following way:

Let $\Sigma_{\mathcal C}=(\ell_1,\ldots,\ell_n)$ denote the collection of linear forms in $R:=\mathbb K[x_1,\ldots,x_k]$ dual to the columns of $G$, considered possibly with repetitions. Let $I({\mathcal C},a)\subset R$ be the ideal generated by all $a$-fold products of the linear forms in $\Sigma_{\mathcal C}$, i.e. $$I({\mathcal C},a)=\langle\{\ell_{i_1}\cdots\ell_{i_a}|1\leq i_1<\cdots<i_a\leq n\}\rangle.$$ Then, by \cite[Exercise 3.25]{DePe}, the minimum distance satisfies $$d=\max\{a|ht(I({\mathcal C},a))=k\}.$$

This result was refined in \cite[Theorem 3.1]{To1} in the following way: $\C$ has minimum distance $d$ if and only if $d$ is the maximal integer such that for any $1\leq a\leq d$, one has $I({\mathcal C},a)=\m^a$, where $\m=\langle x_1,\ldots, x_k\rangle$.

The above result was one of the initial motivations to study the connections between the minimum distance and some invariants coming from commutative/homological algebra. Commutative algebraic techniques have been used extensively in the study of evaluation codes, starting with the work of Hansen \cite{Ha}, yet to our knowledge, the interpretation of minimum distance as a homological invariant started showing up with the famous Cayley-Bacharach theorem and its coding interpretation \cite{GoLiSc}, and subsequently, \cite{To0}, \cite{To2}, and ultimately in \cite{ToVa}. In all these papers, the focus is to get bounds on the minimum distance from the minimal graded free resolution of the ideal of points corresponding to the columns of some generating matrix. The minimum socle degree of this ideal gives a lower bound for the minimum distance; so a half-satisfactory connection (since one does not obtain a formula). The most general result in this direction is \cite[Theorem 2.8]{ToVa}, that presents a lower bound on the minimum distance in terms of the $\alpha$-invariant of the defining ideal of a zero-dimensional fat points scheme.

\medskip

The {\em $\alpha$-invariant} of a finitely generated graded module $M=\oplus_{i\geq 0}M_i$, denoted $\alpha(M)$, is the smallest $i$ for which $M_i\neq 0$; in other words, it is the smallest degree of a generator of $M$. Sometimes, this is called the $a$-invariant.

The ideals generated by $a$-fold products of linear forms do not form a filtration, though \cite[Theorem 3.1]{To1} suggests that they are very close to the $\m$-adic filtration. Because of this, in Section 2 we resort to a somewhat artificial construction that leads to a certain graded module whose $\alpha$-invariant we calculate. It turns out that this module is the Fitting module of a very simple graded module, and we obtain one of the main results in Theorem \ref{result1.1}. In Section 2, we also find connections (see Theorem \ref{result1.3}) with a vector space originally considered in \cite{OrTe}, and explored further in \cite{Be}. In particular, \cite{Be} gives a short exact sequence of these vector spaces under the matroid operations of deletion and restriction. We obtain a similar sequence of Fitting modules for MDS codes (Theorem \ref{result1.4}).

In Section 3 at the beginning we relate the minimum distance with the $\alpha$-invariant of Macaulay inverse systems ideal of the Chow form of a code; the lower bound we find is not very powerful, as it only attains equality in very few examples. This impels us to ask if it is possible to classify all $\C$ for which this bound becomes equality. In the next part we work over $\mathbb F_2$; this allows us to obtain a filtration out of the ideals generated by $a$-fold products of linear forms, and then find a formula for the minimum distance in terms of the $\alpha$-invariant of the positive degree part of an associated graded algebra corresponding to this filtration (see Theorem \ref{result2.5}). In the last subsection we investigate connections with the Orlik-Terao algebra of an arrangement (associated to a linear code $\C$): for codes of dimension 3 we discover an interesting lower bound for the minimum distance in terms of the length of the linear strand of the Orlik-Terao ideal. We end with a result (Proposition \ref{result1.4}) showing that the minimal graded free resolution of the Orlik-Terao algebra is not enough to give complete information about the minimum distance of the code.

\medskip

There are several ways to compute the minimum distance. One method (from \cite{DePe}) is to iteratively calculate heights of ideals generated by $a$-fold products of linear forms, or calculate Gr\"{o}bner bases of such ideals until one obtains a nonempty variety.

Another method comes from linear algebra: the minimum distance of an $[n,k,d]$-linear code with generating matrix $G$, is the number $d$ such that $n-d$ is the maximum number of columns of $G$ that span a $k-1$ dimensional vector (sub)space (see for example, \cite[Remark 2.2]{ToVa}). This second method gives also the geometrical interpretation of minimum distance: assuming that $G$ has no proportional columns, then these columns are $n$ distinct points in $\mathbb P^{k-1}$. Then $n-d$ is the maximum number of these points that fit in a hyperplane.

The generating matrix $G$ of a $[n,k,d]$-linear code $\mathcal C$ naturally determines a matroid $\M(\C)$. The above linear algebraic interpretation suggests that the minimum distance is an invariant of the underlying matroid. This led us to an interesting observation on how to read off the minimum distance by looking at the Tutte polynomial of $\M(\C)$. By definition, {\em the Tutte polynomial} of $\M(\C)$ (or just $\mathcal C$) is $$T_{\mathcal C}(x,y)=\sum_{I\subseteq [n]}(x-1)^{k-r(I)}(y-1)^{|I|-r(I)},$$ where $r(I)$ is the dimension of the linear span of the columns of $G$ indexed by $I$, and $|I|$ is the cardinality of $I$.

There is a strong connection between this polynomial and the invariants of a code, especially the weight-enumerator polynomial; see \cite{JuPe} for a detailed review. The following lemma is included despite the fact that its proof is immediate since we are not aware of this formulation in the literature.

\begin{lem}\label{result1.2}
 The minimum distance is determined by the largest power of $y$ in a term of the form $xy^p$ that appears in $T_{\mathcal C}(x+1,y)$.
 \[
  d=n-p-k+1.
 \] Also the coefficient of this term gives the number of projective codewords of minimum weight.
\end{lem}

\medskip

We are not primarily concerned with developing new computational methods for finding the minimum distance. Our goal is rather to establish connections between the minimum distance with various classical homological invariants. Most of the proofs of our results are not very challenging; we believe that the merit of the notes lies within the interpretations of the minimum distance that we present in the statements of our results. For the background on linear codes we recommend the introductory pages of \cite{JuPe}, and for commutative algebra (including a friendly introduction to filtrations) we suggest \cite{Ei0}.

\section{The Fitting module of a linear code}

We begin with a classical construction in commutative algebra, see \cite[Chapter A2G]{Ei}. Let $M$ be a finitely generated module over a ring $R$. Suppose $M$ has a free presentation $$R^m\stackrel{\phi}\rightarrow R^n\rightarrow M\rightarrow 0.$$

For $1\leq j\leq \min\{m,n\}$, let $I_j(\phi)$ denote the ideal of $R$ generated by the $j\times j$ minors of some matrix representation of $\phi$. By convention, $I_j(\phi)=R$ if $j\leq 0$, and $I_j(\phi)=0$, if $j>\min\{m,n\}$. Then {\em the $k$-th Fitting ideal of $M$} is the ideal $${\rm Fitt}_j(M):=I_{n-j}(\phi).$$ {\em The Fitting module of $M$} is defined to be the $R$-module $${\rm Fitt}(M)=\bigoplus_{j=0}^n\frac{{\rm Fitt}_{n-j}(M)}{{\rm Fitt}_{n-j-1}(M)}.$$

\medskip

We are interested in the following situation: $R=\mathbb K[x_1,\ldots,x_k]$, $\Sigma_{\mathcal C}=(\ell_1,\ldots,\ell_n)$ is the collection of linear forms in $R$ we have seen in the introduction and $M={\rm coker}(\phi)=\frac{R}{\langle\ell_1\rangle}\oplus\cdots\oplus\frac{R}{\langle\ell_n\rangle}$, where
\[
\phi=diag(\Sigma_{\mathcal C})=\left[\begin{array}{ccc} \ell_1 & & \\ &\ddots& \\& &\ell_n\end{array}\right].
\]

One should observe the similarities with \cite[Example A2.56]{Ei}.

For this setup, we denote ${\rm Fitt}(M)$ by ${\rm Fitt}(\mathcal C)$, and call this {\em the Fitting module of the linear code $\mathcal C$}.

\medskip

\begin{thm}\label{result1.1} Let $\mathcal C$ be an $[n,k,d]$-linear code. Let $\Sigma_{\mathcal C}$ be the collection of $n$ linear forms in $R=\mathbb K[x_1,\ldots,x_k]$ dual to the columns of some generating $k\times n$ matrix of $\mathcal C$. If $\m=\langle x_1,\ldots,x_k\rangle$, then the minimum distance of $\mathcal C$ satisfies $$d=\alpha(\m{\rm Fitt}(\mathcal C))-1.$$
\end{thm}
\begin{proof} We have $\phi=diag(\Sigma_{\mathcal C})$ and $M={\rm coker}(\phi)$. Then $${\rm Fitt}_{n-j}(M)=I_j(\phi)=I({\mathcal C},j),$$ which is the ideal generated by $j$-fold products we have seen in the introduction.

This is a homogeneous ideal generated in degree $j$. Also it is clear that $I({\mathcal C},j+1)\subset I({\mathcal C},j)$, for $j=1,\ldots,n-1$. Then $${\rm Fitt}(\mathcal C)=\frac{R}{I({\mathcal C},1)}\oplus \frac{I({\mathcal C},1)}{I({\mathcal C},2)}\oplus\cdots\oplus\frac{I({\mathcal C},d-1)}{I({\mathcal C},d)}\oplus \frac{I({\mathcal C},d)}{I({\mathcal C},d+1)}\oplus\cdots\oplus \frac{I({\mathcal C},n-1)}{I({\mathcal C},n)}\oplus I({\mathcal C},n).$$

As mentioned in the introduction, $I({\mathcal C},j)=\m^j, j=1,\ldots,d$. See \cite[Theorem 3.1]{To1}. So
\[
\m\cdot{\rm Fitt}(\mathcal C)=0\oplus \cdots\oplus 0\oplus \frac{\m^{d+1}}{I({\mathcal C},d+1)}\oplus\cdots\oplus \frac{\m\cdot I({\mathcal C},n-1)}{I({\mathcal C},n)}\oplus \m I({\mathcal C},n),
\]
since $I({\mathcal C},j+1)\subset \m\cdot I({\mathcal C},j),$ for all $j\geq 0$.

In the proof of \cite[Proposition 2.4]{To1} there exists $f\in R_d\setminus I({\mathcal C},d+1)$, with $\m\cdot f\in I({\mathcal C},d+1)$. This $f$ is in fact an element of degree $d$ in the saturation with respect to $\m$ of $I({\mathcal C},d+1)$. Avoiding such elements, let $g\in R_d\setminus sat(I({\mathcal C},d+1))$. Such an element exists since by \cite[Lemma 2.2]{To1} this saturated ideal is intersection of codimension $k-1$ prime ideals, and therefore itself has codimension $k-1$; or by using \cite[Lemma 2.1]{To1} and the fact that an ideal and its saturation have the same codimension. But if $\m\cdot g\subset I({\mathcal C},d+1)$, then $g\in sat(I({\mathcal C},d+1))$, contradiction. So $$\displaystyle\alpha(\frac{\m^{d+1}}{I({\mathcal C},d+1)})=d+1.$$

If $M_i,i=1,\ldots,n$ are some graded $R$-modules, then $N:=\oplus M_i$ becomes a graded $R$-module with the natural grading $N_k=\{(m_1,\ldots,m_n)|m_i\in (M_i)_k \mbox{ for all }1\leq i\leq n\}$. Therefore ${\rm Fitt}(\mathcal C)$, and hence $\m\cdot{\rm Fitt}(\mathcal C)$ become graded $R$-modules (the latter being a submodule of the former). Therefore $\alpha(\m{\rm Fitt}(\mathcal C))=d+1$.
\end{proof}

\medskip

In \cite{Be} an interesting vector space is presented; we will adjust everything to our notation. Let $\mathcal C$ be an $[n,k,d]$-linear code with $\displaystyle \Sigma_{\mathcal C}=(\ell_1,\ldots,\ell_n)\subset R:=\mathbb K[x_1,\ldots,x_k]$.

For any $I\subset [n]$, denote $\displaystyle\ell_I=\Pi_{i\in I}\ell_i$, with the convention that $\ell_{\emptyset}=1$. Let $P(\mathcal C)$ be the $\mathbb K$-vector subspace of $R$ spanned by $\ell_I$, for all $I\subset [n]$. Then one has a decomposition:

$$P(\mathcal C)=\bigoplus_{0\leq u\leq v\leq n}P(\mathcal C)_{u,v},$$ where
\[
 P(\mathcal C)_{u,v}={\rm Span}_{\mathbb K}\{\ell_I|\dim_{\mathbb K}({\rm Span}_{\mathbb K}\{\ell_j,j\in[n]-I\})=u\mbox{ and }v=n-|I|\}.
\]

We have $0\leq u\leq k$ since ${\rm Span}_{\mathbb K}\{\ell_j,j\in[n]-I\})$ is a subspace of $\mathbb K^k$.

Using $P(\mathcal C)$, \cite[Theorem 1.1]{Be} shows that the Tutte polynomial satisfies: $$T_{\mathcal C}(x,y)=\sum_{0\leq u\leq v\leq n} (x-1)^{k-u}y^{v-u}\dim_{\mathbb K}P(\mathcal C)_{u,v}.$$

The connection between $P(\C)$ and ${\rm Fitt}(\C)$ is the following:

\begin{thm}\label{result1.3} There is an isomorphism of $\mathbb K$-vector spaces:
\[
P(\mathcal C)\simeq {\rm Fitt}(\mathcal C)\otimes_R\mathbb K.
\]
\end{thm}
\begin{proof} One has $\displaystyle{\rm Fitt}(\mathcal C)\otimes_R{\mathbb K}=\frac{{\rm Fitt}(\mathcal C)}{\m{\rm Fitt}(\mathcal C)}$. The later $\mathbb K$-vector space is isomorphic to
\[
\frac{R}{\m}\oplus \frac{\m}{\m^2} \oplus\cdots\oplus\frac{\m^d}{\m^{d+1}}\oplus \frac{I({\mathcal C},d+1)}{\m I({\mathcal C},d+1)}\oplus\cdots\oplus \frac{I({\mathcal C},n-1)}{\m I({\mathcal C},n-1)} \oplus \frac{I({\mathcal C},n)}{\m I({\mathcal C},n)}.
\]

For all $j=1,\ldots,n$,
\[
 \displaystyle\frac{I({\mathcal C},j)}{\m I({\mathcal C},j)}=I({\mathcal C},j)_j={\rm Span}_{\mathbb K}\{\ell_{i_1}\cdots\ell_{i_j}|1\leq i_1<\cdots<i_j\leq n\}.
\]

With $I({\mathcal C},0)_0=R_0=\mathbb K$, the isomorphism is clear.
\end{proof}

\begin{rem}\mbox\\
\begin{itemize}
\item If $|I|=n-v$, then $|\{j; j\in [n]-I\}|=v$, and therefore $\dim_{\mathbb K}({\rm Span}_{\mathbb K}\{\ell_j,j\in[n]-I\})\leq v$. This leads to $\displaystyle\bigoplus_{0\leq u\leq v}P(\mathcal C)_{u,v}=I({\mathcal C},n-v)_{n-v}$, for all $v=0,\ldots,n$.

\item  In the spirit of Lemma \ref{result1.2}, if $p=n-d-k+1$, then the coefficient of $xy^p$ in $T_{\mathcal C}(x+1,y)$ is $\dim_{\mathbb K}P(\mathcal C)_{k-1,n-d}$. The following is an explanation as to why this dimension equals the number of projective codewords of minimum weight.

From \cite{To1}, projective codewords of minimum weight are in one-to-one correspondence with the points of the zero-dimensional variety $V(I({\mathcal C},d+1))\subset\mathbb P^{k-1}$. All the associated primes of the defining ideal of this variety are prime ideals of codimension $k-1$, generated (not minimally) by $n-d$ linear forms. So if $P\in V(I({\mathcal C},d+1))$ with ideal $I_P=\langle \ell_{i_{d+1}},\ldots,\ell_{i_n}\rangle$, then none of the complementary linear forms $\ell_{i_1},\ldots,\ell_{i_d}$ vanishes at $P$. Denote $J(P)=\{i_1,\ldots,i_d\}$.

Suppose one has $P_1,\ldots,P_s$ such points, with $\ell_{J(P_1)}(P_1),\ldots,\ell_{J(P_s)}(P_s)\neq 0$, and suppose $$c_1\ell_{J(P_1)}+\cdots +c_s\ell_{J(P_s)}=0,$$ for some $c_i\in\mathbb K$. Evaluating this at $P_i$, since $\ell_{J(P_j)}(P_i)=0, i\neq j$, one has that $c_i=0$, giving that $\ell_{J(P_1)},\ldots,\ell_{J(P_s)}$ are linearly independent.

\item The coding theoretical equivalent of \cite[Theorem 5.1]{Be} is \cite[Theorem 3.1]{To1}.
\end{itemize}
\end{rem}

\subsection{Deletion-restriction}

Let $\mathcal C$ be an $[n,k,d]$-linear code with generating matrix $G$. {\em The puncturing} of $\mathcal C$ at the $i$-th column of $G$ is the linear code denoted $\mathcal C\setminus i$ with generating matrix obtained from $G$ by removing the $i$-th column. If the dimension of $\mathcal C\setminus i$ is $k-1$, then $d=1$, which is a very particular situation (in this case $i$ is called a {\em coloop}). So, if $i$ is not a coloop, then the parameters of $\mathcal C\setminus i$ are $[n-1,k,d(i)]$, where the minimum distance $d(i)$ equals $d$ or $d-1$.

We are interested in puncturing $\mathcal C$ at a column $i$ of $G$ that has the last entry $\neq 0$. Doing row operations on $G$ one can assume that this column is of the form $\left[\begin{array}{cccc} 0 & \cdots & 0 &1\end{array}\right]^T$.

Let $\mathcal C/i$ be {\em the shortening} of $\mathcal C$ at the above $i$-th column of $G$. This is a linear code with generating matrix $G_i$ obtained from $G$, by deleting the last row and this $i$-th column. Geometrically we restrict the hyperplanes $V(\ell_i)$ to the hyperplane $V(x_k)$. The linear code $\mathcal C/i$ has parameters $[n-1,k-1,d_i]$. In general, $d_i\geq d(i)$.

From now on, let us assume $i=n$, and we denote $\mathcal C\setminus n$ and $d(n)$ by $\mathcal C'$ and $d'$, respectively; and we denote $\mathcal C/n$ and $d_n$ by $\mathcal C''$ and $d''$, respectively.

At the level of Fitting ideals one can immediately show that \begin{equation}
I({\mathcal C},a)=x_kI({\mathcal C'},a-1)+I({\mathcal C'},a), a=1,\ldots,n \label{eq:delres1}
 \end{equation}
and
\begin{equation}
I({\mathcal C},a)+\langle x_k\rangle=I({\mathcal C'},a)+\langle x_k\rangle= I({\mathcal C''},a)+\langle x_k\rangle. \label{eq:delres2}
\end{equation}

In matroidal terms, ``puncturing'' and ``shortening'' of $\C$ correspond to ``deletion'' and ``restriction (contraction)'' of $\M(\C)$.

\subsubsection{MDS codes.} \cite[Lemma 6.2]{Be} shows that the vector space $P(\C)$ behaves very well under deletion and restriction, whereas our Fitting modules behave in a more complicated manner. However, if $\C$ is an MDS code (i.e., $d=n-k+1$), we can obtain a similar short exact sequence of $R:=\mathbb K[x_1,\ldots,x_k]$-modules. Here, we have $R'':=\mathbb K[x_1,\ldots,x_{k-1}]=R/\langle x_k\rangle$, which gives the Fitting module of the restriction a natural $R$-module structure.

It is not difficult to see that if $\C$ is an MDS code, then $\C'$ and $\C''$ are also MDS codes. Combinatorially, $\C$ is MDS if and only if the underlying matroid $\M(\C)$ is isomorphic to the uniform matroid $U_{k,n}$ (see \cite[page 49]{JuPe}).

\medskip

In order to obtain our result, we have to look at {\em star configurations} in $\mathbb P^{k-1}$ (see \cite{GHM} for relevant details). For $1\leq c\leq k-1$, consider the ideal $I_{V_c}$ of the star configuration $V_c$:
\[
 V_c= \bigcup_{1\leq i_1<\cdots< i_c\leq n} H_{i_1}\cap \cdots \cap H_{i_c}, \quad I_{V_c}=\bigcap_{1\leq i_1<\cdots< i_c\leq n} \langle \ell_{i_1}, \ldots, \ell_{i_c} \rangle.
\]

\begin{lem}{\rm (}\cite[Proposition 2.9]{GHM}{\rm )}\label{lemma}
Let $\C$ be an MDS code with parameters $[n,k,n-k+1]$. Then, for $d+1=n-k+2\leq j\leq n$,
\begin{enumerate}
  \item $I({\C}, j)=I_{V_{n-j+1}}$ and
  \item if $c=n-j+1$, the Hilbert series satisfies
  \[
HS(R/I_{V_c},t)=\frac{\sum_{u=0}^{n-c} \binom{c-1+u}{c-1}t^u}{(1-t)^{k-c}}.
\]
\end{enumerate}

\end{lem}
\begin{proof}
From the beginning of Section 2 in \cite{To1}, it follows that $I({\C}, j)\subseteq I_{V_{n-j+1}}$.

The proof of \cite[Proposition 2.9(4)]{GHM} makes use of the equation \ref{eq:delres1}, to show by induction that $I_{V_{n-j+1}}\subseteq I({\C},j)$.

The second part is immediate from \cite{GHM}.
\end{proof}

\begin{thm}\label{result1.4} Let $\C$ be an $[n,k]$ MDS code, and let $\ell\in\Sigma_{\C}$. Let $\C'$ and $\C''$ be the deletion and restriction at $\ell$, respectively. Then one has a short exact sequence of $R$-modules $$0\longrightarrow {\rm Fitt}(\C')(-1)\stackrel{\cdot \ell}\longrightarrow {\rm Fitt}(\C)\longrightarrow {\rm Fitt}(\C'')\longrightarrow 0.$$
\end{thm}
\begin{proof} Assume $\ell=x_k$, and denote $I_a:=I({\C},a)\subset R,I_a':=I({\C'},a)\subset R$ and $I_a'':=I({\C''},a)\subset R''$. As we have discussed earlier, $\C$ is an $[n,k,d]$-linear code, $\C'$ is an $[n-1,k,d-1]$-linear code, and $\C''$ is an $[n-1,k-1,d]$-linear code with $d=n-k+1$.

Equations \ref{eq:delres1} and \ref{eq:delres2} give rise to the following short exact sequence of $R$-modules, determined by multiplication by $x_k$:
\[
 0\longrightarrow \frac{R}{I_{a+1}:x_k}(-1) \longrightarrow\frac{R}{I_{a+1}}\longrightarrow\frac{R}{I_{a+1}+\langle x_k\rangle }\simeq\frac{R''}{I_{a+1}''}\longrightarrow 0,
\] for all $a=0,\ldots,n$.

For $a+1\geq d+1$, computing the Hilbert series of the first term via the short exact sequence above, using Lemma \ref{lemma} (with $c=c'=n-a$ and $c''=n-a-1$), one has

\begin{eqnarray*}
 t\cdot HS(R/(I_{a+1}:\langle \ell\rangle),t)& = & HS(R/I_{a+1},t)-HS(R''/I_{a+1}'',t)\\
  & = & \frac{\sum_{u=0}^{a} \left( \binom{n-a-1+u}{n-a-1}-\binom{n-a-2+u}{n-a-2}\right ) t^u}{(1-t)^{k-a-1}}\\
    & = & \frac{\sum_{u=1}^{a} \binom{n-a-2+u}{n-a-1} t^u}{(1-t)^{k-a-1}}\\
        & = & \frac{t\sum_{u=0}^{a-1} \binom{n-a-1+u}{n-a-1} t^u}{(1-t)^{k-a-1}}=t\cdot HS(R/I_a',t).\\
\end{eqnarray*}

Since $I_a'\subseteq I_{a+1}:x_k$, from above we get in fact the equality $I_a'=I_{a+1}:x_k$.

\medskip

If $a<d$, then, by \cite[Theorem 3.1]{To1}, $I_{a+1}=\langle x_1,\ldots,x_k\rangle^{a+1}$, $I_a=\langle x_1,\ldots,x_k\rangle^a$, and $I_{a+1}''=\langle x_1,\ldots,x_{k-1}\rangle^{a+1}$.

Since $\langle x_1,\ldots,x_k\rangle^{a+1}:x_k=\langle x_1,\ldots,x_k\rangle^a$ and $\langle x_1,\ldots,x_k\rangle^{a+1}+\langle x_k\rangle= \langle x_1,\ldots,x_{k-1}\rangle^{a+1}+\langle x_k\rangle$, for all $a=0,\ldots,n$ we have the short exact sequence of $R$-modules:
\[
 0\longrightarrow \frac{R}{I_a'}(-1) \longrightarrow\frac{R}{I_{a+1}}\longrightarrow\frac{R''}{I_{a+1}''}\longrightarrow 0.
\]

Snake Lemma applied to the surjective map of complexes
\[
\begin{array}{ccccccccc}
0&\longrightarrow&\frac{R}{I_{a-1}'}(-1)& \longrightarrow&\frac{R}{I_{a}}&\longrightarrow&\frac{R''}{I_{a}''}&\longrightarrow& 0\\
 &               &\downarrow&                            &\downarrow&                     &\downarrow&&\\
0&\longrightarrow&\frac{R}{I_{a}'}(-1)& \longrightarrow&\frac{R}{I_{a+1}}&\longrightarrow&\frac{R''}{I_{a+1}''}&\longrightarrow& 0
\end{array}
\]
leads to the short exact sequence of $R$-modules

\[
 0\longrightarrow \frac{I_{a-1}'}{I_a'}(-1) \longrightarrow\frac{I_a}{I_{a+1}}\longrightarrow\frac{I_a''}{I_{a+1}''}\longrightarrow 0,
\] for all $a=0,\ldots,n$.

Taking the direct sum of all them, one obtains the claimed statement.
\end{proof}

\section{Other connections and results}

\subsection{Inverse systems}

For this subsection the base field $\mathbb K$ is a field of zero or sufficiently large characteristic.

Let $\partial(V):=\mathbb K[\partial_{1},\ldots,\partial_{k}]$ where $V=\K^k$ and $\partial_i$ is a short hand notation for $\partial/\partial_{x_k}$. $\partial(V)$ acts on $R=\K[x_1,\dots, x_k]$ in the usual way where $\partial_i(x_j)=\delta_{i,j}$, extended by linearity and the Leibniz rule.

%$$x_1^{i_1}\cdots x_k^{i_k}\circ y_1^{j_1}\cdots y_k^{j_k}= \frac{\partial^{i_1+\cdots+i_k}}{\partial y_1^{i_1}\cdots \partial y_k^{i_k}}(y_1^{j_0}\cdots y_k^{j_k}),$$ extended by linearity.

Let $P\in R_r$. Then the set $Ann(P):=\{\theta \in \partial(V) | \theta P=0\}$ is a homogeneous ideal of $\partial(V)$; furthermore $\partial(V)/Ann(P)$ is Artinian Gorenstein. As consequences of this result (also known as ``Macaulay's Inverse Systems Theorem'') one has

\begin{enumerate}
  \item The Castelnuovo-Mumforde regularity ${\rm reg}(\partial(V)/Ann(P))=\deg(P)=r$.
  \item The Hilbert function of $\partial(V)/Ann(P)$ in degree $i$ (i.e., $HF(\partial(V)/Ann(P),i)=\dim_{\mathbb K}(\partial(V)/Ann(P))_i$) equals the dimension of the vector space spanned by the partial derivatives of order $i$ of $P$.
  \item From the Gorenstein property one can obtain the symmetry of the Hilbert function of $R/Ann(P)$, i.e., $HF(\partial(V)/Ann(P),i)=HF(\partial(V)/Ann(P),r-i)$, for all $i=0,\ldots,r$.
\end{enumerate}

For details about inverse systems \cite{Ge} is a good source.

\medskip

Let $\mathcal C$ be an $[n,k,d]$- linear code with $\Sigma_{\mathcal C}=(\ell_1,\ldots,\ell_n)\subset R$. Let $\displaystyle {\rm cf}(\mathcal C)=\Pi_{i=1}^n\ell_i$, be the {\em Chow form} of $\mathcal C$. In the next result, we see again the $\alpha$-invariant showing up.

\begin{prop}\label{result3.1} If $char(\mathbb K)>n$, or $char(\mathbb K)=0$, then $$d\geq\alpha(Ann({\rm cf}(\mathcal C)))-1.$$
\end{prop}
\begin{proof} Denote $I:=Ann({\rm cf}(\mathcal C))$ and $\alpha:=\alpha(I)$.

For all $j=0,\ldots,\alpha-1$, one has $$HF(\partial(V)/I,j)={{k-1+j}\choose{k-1}}.$$

From the symmetry of the Hilbert function of $\partial(V)/I$, we have that for all $j=0,\ldots,\alpha-1$ $$HF(\partial(V)/I,n-j)={{k-1+j}\choose{k-1}},$$ and therefore, from item (2) above, $$\dim_{\mathbb K}{\rm Span}(D^{n-j}({\rm cf}(\mathcal C)))={{k-1+j}\choose{k-1}}.$$

Since $\displaystyle {\rm cf}(\mathcal C)=\prod_{i=1}^n\ell_i$, then $${\rm Span}(D^{n-j}({\rm cf}(\mathcal C)))\subseteq {\rm Span}(\{\ell_{i_1}\cdots \ell_{i_j}\}_{1\leq i_1<\cdots<i_j\leq n}), $$ and therefore $$HF(I({\mathcal C},j),j)={{k-1+j}\choose{k-1}}=HF(\langle x_1,\ldots,x_k\rangle^j,j).$$

Since $I({\mathcal C},j)$ is generated by forms of degree $j$, we conclude that $$I({\mathcal C},j)=\langle x_1,\ldots,x_k\rangle^j,$$ for all $0\leq j\leq \alpha-1$. From \cite[Theorem 3.1]{To1}, one obtains $d\geq \alpha-1$.
\end{proof}

\begin{rem} One can prove Proposition \ref{result3.1} without the use of the commutative algebraic machinery. Consider ${\rm cf}(\mathcal C)$ the Chow form of $\mathcal C$. A codeword of minimum weight of $\mathcal C$ corresponds to a point through which $m:=n-d$ linear forms of $\Sigma_{\mathcal C}$ will pass. Suppose this point is $Q:=[0,\ldots,0,1]\in\mathbb P^{k-1}$, and suppose $\ell_1(Q)=\cdots=\ell_m(Q)=0$, and hence $\ell_1,\ldots,\ell_m\in\mathbb K[x_1,\ldots,x_{k-1}]$. Then $${\rm cf}(\mathcal C)=(\ell_1\cdots\ell_m)\cdot(\ell_{m+1}\cdots\ell_n).$$

One has $$\partial_k^{d+1}({\rm cf}(\mathcal C))=(\ell_1\cdots\ell_m)\cdot \partial_k^{d+1} (\ell_{m+1}\cdots\ell_n).$$ The later equals 0, as $\deg(\ell_{m+1}\cdots\ell_n)=d$. So $\partial_k^{d+1}\in Ann({\rm cf}(\mathcal C))$.

\end{rem}

\begin{exm} One question is to analyse the arrangements for which equality in Theorem \ref{result3.1} holds true. For some linear codes (even MDS codes) this happens, but for most of them it doesn't.

Consider $\mathcal C_1$ defined by the generating matrix $G_1$. It has minimum distance $d=2=4-3+1$, so it is an MDS code. Computations with \cite{GrSt} give that $\alpha(Ann(xyz(x+y+z)))=3=d+1$.

$$G_1=\left[\begin{array}{cccc}1&0&0&1\\ 0&1&0&1\\0&0&1&1\end{array}\right],\quad G_2=\left[\begin{array}{ccccc}1&0&0&1&1\\ 0&1&0&1&2\\0&0&1&1&5\end{array}\right]$$

\medskip

Consider $\mathcal C_2$ with generating matrix $G_2$. It has minimum distance $d=3=5-3+1$, so it is an MDS code. Computations with \cite{GrSt} give that $\alpha(Ann(xyz(x+y+z)(x+2y+5z)))=3<d+1$.
\end{exm}
\vskip .1in

\subsection{The case of binary codes}

Let $S:=\mathbb K[y_1,\ldots,y_n]$ and consider the ring epimorphism $$\gamma:S\rightarrow R, \gamma(y_i)=\ell_i, 1\leq i\leq n.$$ The kernel of $\gamma$, denoted here $F({\mathcal C})$, is an ideal of $S$, minimally generated by $n-k$ linear forms in $S$. This is often called the \emph{relation space} of $\Sigma_{\mathcal C}$. The matrix of the coefficients of the standard basis of $F({\mathcal C})$ is a $(n-k)\times n$ matrix which is the generating matrix of the dual code of $\mathcal C$. Standard refers to the fact that transpose of this matrix is the parity-check matrix of $\mathcal C$. If $G=\left[I_k|P\right]$, then $G^{\perp}=\left[-P^T|I_{n-k}\right]$.

It is well known that a vector ${\bf w}=(w_1,\ldots,w_n)$ is in $\mathcal C$, if and only if the point $(w_1,\ldots,w_n)$ belongs to the linear variety with defining ideal $F({\mathcal C})$. Also, a vector ${\bf w}=(w_1,\ldots,w_n)$ has weight $\leq a-1$, if and only if all the distinct $a$ products of its entries vanish. Equivalently, the point $(w_1,\ldots,w_n)$ is in the variety with defining ideal $$I({\rm Y},a):=\langle \{y_{i_1}\cdots y_{i_a}|1\leq i_1<\cdots<i_a\leq n\}\rangle,$$ which is the ideal of $S$ generated by the $a$-fold products of ${\rm Y}:=(y_1,\ldots,y_n)$.

With this, one obtains immediately that $\mathcal C$ has minimum distance $d$ if and only if it is the largest integer such that

\begin{equation}
\sqrt{F({\mathcal C})+I({\rm Y},a)}=\langle y_1,\ldots,y_n\rangle,
\label{equation3}
\end{equation} for all $a=1,\ldots,d$. It is clear from this that if one considers the linear code $\mathcal D$ of length $2n-k$ and dimension $n$ with generating matrix $\left[(G^{\perp})^T|I_n
\right ]$, and then shortens it to the first $n-k$ columns, the new code has minimum distance $d$ (in fact, this code is up to permutations and rescaling the same as $\mathcal C$).

\medskip

From now on our base field $\mathbb K$ will be $\mathbb F_2$. Any element of $\mathbb F_2$ satisfies the field equation $X^2=X$, so we are going to consider the finite dimensional $\mathbb F_2$-algebra
\[
\mathbb S:=S/\langle y_1^2-y_1,\ldots,y_n^2-y_n\rangle.
\]
Let $S\to \mathbb{S}: f\mapsto \bar{f}$ be the natural reduction map.

\begin{prop}\label{result2.2} The linear code $\mathcal C$ has minimum distance $d$ if and only if in $\mathbb S$ one has $$\bar{F}({\mathcal C})+\bar{I}({\rm Y},a)=\langle \bar{y}_1,\ldots,\bar{y}_n\rangle,$$ for all $a=1,\ldots,d$,
where $d$ is maximal with this property.
\end{prop}

\begin{proof} Formula \ref{equation3} says that $y_i^{n_i}\in F({\mathcal C})+I({\rm Y},a), i=1,\ldots,n$, for some $n_i\geq 1$. But in $\mathbb S$, we have $\bar{y}_i^{n_i}=\bar{y}_i$. Hence the result.
\end{proof}

%\medskip

\subsubsection{Standard filtration.} The algebra $\mathbb S$ is a filtered algebra, with ``standard'' filtration given by the $\mathbb S$-modules: $$\mathbb S_j:=\mathbb S/\bar{I}(Y,j+1),j=0,\ldots,n.$$ Indeed one has $$\mathbb F_2=\mathbb S_0\subset \mathbb S_1\subset\cdots\subset\mathbb S_n=\mathbb S,$$ with $\mathbb S_a\cdot\mathbb S_b\subseteq\mathbb S_{a+b}$.

Proposition \ref{result2.2} immediately implies the following.

\begin{cor}\label{result2.3} $\mathcal C$ has minimum distance $d$ if and only if $d$ is the maximal integer such that for any $a=1,\ldots,d$, one has $$\frac{\mathbb S}{\bar{F}({\mathcal C})}\otimes_{\mathbb S}\mathbb S_a=\mathbb F_2,$$ as $\mathbb S$-modules.
\end{cor}

%\medskip

\subsubsection{Another filtration.} On $\mathbb S$ there is also a filtration given by the ideals $\bar{I}({\rm Y},a)$:
$$\mathcal F_{\bullet}:\, 0\subset\underbrace{\bar{I}({\rm Y},n)}_{\mathcal F_0}\subset\underbrace{\bar{I}({\rm Y},n-1)}_{\mathcal F_1}\subset\cdots\subset \underbrace{\bar{I}({\rm Y},1)}_{\mathcal F_{n-1}}\subset\underbrace{\bar{I}({\rm Y},0)}_{\mathcal F_n}=\mathbb S.$$

\begin{lem}\label{result2.4} $\mathcal F_{\bullet}$ is a filtration on $\mathbb S$, meaning that
\begin{enumerate}
  \item $\displaystyle\mathbb S=\cup_i \mathcal F_i$.
  \item $\mathcal F_a\cdot\mathcal F_b\subset \mathcal F_{a+b}$.
\end{enumerate}
\end{lem}
\begin{proof} The first statement is obvious.

For the second, let $\bar{y}_{i_1}\cdots\bar{y}_{i_{n-a}}\in \mathcal F_a$ and $\bar{y}_{j_1}\cdots\bar{y}_{j_{n-b}}\in \mathcal F_b$. Suppose that $|\{i_1,\ldots,i_{n-a}\}\cap\{j_1,\ldots,j_{n-b}\}|=c$, with $0\leq c\leq\min\{n-a,n-b\}$.

Since in $\mathbb S$, one has $\bar{y}_u^2=\bar{y}_u$, for any $u$, one obtains $$(\bar{y}_{i_1}\cdots\bar{y}_{i_{n-a}})\cdot (\bar{y}_{j_1}\cdots\bar{y}_{j_{n-b}})\in \bar{I}({\rm Y},2n-a-b-c).$$ Since $2n-a-b-c\geq n-a-b$, one has that $$\bar{I}({\rm Y},2n-a-b-c)\subseteq \bar{I}({\rm Y},n-a-b)=\mathcal F_{a+b}.$$ Hence the second statement is shown.
\end{proof}

On $\mathbb S/\bar{F}({\mathcal C})$ one has the induced filtration $\mathcal F_{\bullet}(\mathbb S/\bar{F}({\mathcal C}))=(\bar{F}({\mathcal C})+\mathcal F_{\bullet})/\bar{F}({\mathcal C})$. So one can consider {\em the associated graded module} $${\rm gr}_{\mathcal F_{\bullet}}(\mathbb S/\bar{F}({\mathcal C})):=\bigoplus_{i=0}^{n}\frac{\mathcal F_i(\mathbb S/\bar{F}({\mathcal C}))}{\mathcal F_{i-1}(\mathbb S/\bar{F}({\mathcal C}))}.$$

The main result of this subsection is the following homological interpretation of the minimum distance, connecting once again to the $\alpha$ invariant of graded modules.

\begin{thm}\label{result2.5} With the above notations $$\alpha\left({\rm gr}_{\mathcal F_{\bullet}}(\mathbb S/\bar{F}({\mathcal C}))_+\right)=n-d.$$
\end{thm}
\begin{proof} The proof is immediate from observing that for all $u=n-1,\ldots,n-d$ $$\mathcal F_u(\mathbb S/\bar{F}({\mathcal C}))=\frac{\bar{F}({\mathcal C})+\bar{I}({\rm Y},n-u)}{\bar{F}({\mathcal C})}=\frac{\langle \bar{y}_1,\ldots, \bar{y}_n\rangle}{\bar{F}({\mathcal C})},$$ the second equality being the result of Proposition \ref{result2.2}.

Taking appropriate quotients, one obtains the result.
\end{proof}

\vskip .1in

\subsection{The Orlik-Terao algebra}

Let $\mathbb K$ be any field, and let $\mathcal C$ be an $[n,k,d]$-linear code with $\Sigma_{\C}=(\ell_1,\ldots,\ell_n)\subset R:=\mathbb K[x_1,\ldots,x_k], gcd(\ell_i,\ell_j)=1,i\neq j$. Define a ring epimorphism
\begin{equation}
\phi:\, S:=\mathbb K[y_1,\ldots,y_n]\rightarrow \mathbb K[1/\ell_1,\ldots, 1/\ell_n],\, \phi(y_i)=1/\ell_i.\label{eqn:otmap}
\end{equation}

{\em The Orlik-Terao algebra} is ${\rm OT}(\mathcal C):=\K[y_1,\ldots,y_n]/\ker(\phi)$. $\ker(\phi)$ is called {\em the Orlik-Terao ideal}, and it is denoted ${\rm IOT}(\mathcal C)$.

It is well known that ${\rm IOT}(\mathcal C)$ is generated by
\begin{equation}
 \displaystyle\partial(a_1y_{i_1}+\cdots+a_uy_{i_u}):=\sum_{j=1}^{u}a_jy_{i_1}\cdots\widehat{y_{i_j}}\cdots y_{i_u},\label{eqn:r_C}
\end{equation}
where $a_1y_{i_1}+\cdots+a_uy_{i_u}, a_j\neq 0$ is an element in the relations space $F({\mathcal C})$ we have seen at the beginning of the previous subsection. Properties of ${\rm OT}(\C)$ can be found, for example in \cite{DeGaTo}, and the citations therein.

\begin{rem}
 Recall that $\C=\Ima G=\{\mathbf{x} G: \mathbf{x}\in \K^k\}$ and $V(F(\C))=\ker G= \{ \mathbf{y}\in \K^n: G\mathbf{y}=0\}$ have dimensions $k$ and $n-k$, respectively. As mentioned already, $F(\C)=\C^{\perp}$.
\end{rem}

\begin{equation}\label{eq:}
\xymatrix{
 \K^k \ar[d]_{id}\ar[r]^G & \K^n \ar@{-->}[d]^{crem} \\
 \K^k \ar[r]_{\phi^*} & \K^n
}
\end{equation}

Let $T=(\K^*)^{\times}$ be the $n$-dimensional torus. The vertical map is the Cremona transformation $T\to T$, ${\bf{y}}=(y_1,\dots, y_n) \mapsto {\bf{y}}^{-1}=(y_1^{-1},\dots, y_n^{-1})$. The linear space of the code $\C$ is the vanishing of the maximal minors of the following augmented matrix.

\[
\C=\{ {\bf{y}} : \rank \left[\begin{array}{ccc} &G& \\\hline y_1&\cdots & y_n\end{array}\right] = k\}
\]

The variety defined by the OT is called the \emph{reciprocal plane} which is the closure of the image of $\C\cap T$ under the transformation. Alternatively, \cite[Proposition 2.6]{DeGaTo} implies that the OT ideal can be obtained as the colon ideal ${\rm IOT}(\C) = (I : y_1\cdots y_n)$, where $I$ is generated by the maximal minors of the following matrix.

\[
 \left[\begin{array}{cccc}a_{11}y_1& \cdots &a_{1n}y_n\\ \vdots  & \ddots & \vdots \\ a_{k1}y_1& \cdots  &a_{kn}y_n \\
 \hline
 1 & \cdots & 1\end{array}\right]
\]
The connection between the OT ideals of $\C$ and $\C^{\perp}$ is that they come from reciprocal planes of orthogonal spaces.

In order to calculate the minimum distance $d$, we are interested in codewords (so points in $V(F({\C}))$) with lots of zero entries, $n-d$ to be precise. From the remarks above, $\C$ and ${\rm OT}(\C)$ interact very well while inside the torus, hence to find points with zero coordinates one needs to investigate the fiber over zero of the Orlik-Terao algebra, also called ``the relative Orlik-Terao algebra'' (see \cite{DeGaTo} for more details). It is still not clear how one can obtain a nice description of $d$ from this approach; this path will be the focus of a future study.

\medskip

\subsubsection{Length of linear strand} Let $M$ be a finitely generated graded $S$-module with $\alpha(M)=\alpha$. Suppose that the minimal graded free resolution of $M$ has the shape $$0\rightarrow {\bf F}_n\rightarrow\cdots\rightarrow {\bf F}_{\delta+1}\rightarrow S^{n_{\delta}}(-(\alpha+\delta))\oplus {\bf F}_{\delta}\rightarrow\cdots\rightarrow S^{n_0}(-\alpha)\oplus {\bf F}_0\rightarrow M\rightarrow 0,$$ where $\alpha({\bf F}_j)\geq \alpha+j+1$, or they are zero. Then $\delta$ is called {\em the length of the linear strand} of $M$.

\medskip

From now on we will assume $k=3$. So we are looking at $\mathcal C$ with parameters $[n,3,d]$.

\begin{prop}\label{result3.2} Let $\mathcal C$ be an $[n,3,d]$-linear code with nonproportional linear forms of $\Sigma_{\mathcal C}$. Let $\delta$ denote the length of the linear strand of the Orlik-Terao ideal ${\rm IOT}(\C)$, which is assumed not to be the zero ideal. Then
\begin{enumerate}
  \item If $\alpha({\rm IOT}(\C))=3$, then $\C$ is an MDS code.
  \item If $\alpha({\rm IOT}(\C))=2$, then $d\geq n-\delta-3$.
\end{enumerate}
\end{prop}
\begin{proof} If $\alpha({\rm IOT}(\C))=3$, then any 3 of the linear forms in $\Sigma_{\mathcal C}$ are linearly independent, hence the claim (1).

Suppose $\alpha({\rm IOT}(\C))=2$. Denote $\mathcal A$ to be the rank 3 arrangement of lines in $\mathbb P^2$ with lines $H_i=V(\ell_i)$. Let $X\in L_2(\mathcal A)$, and suppose $X=H_1\cap\cdots\cap H_p$, with $p\geq 3$, maximal as possible (hence $p=n-d$).

By \cite[Theorem 3.1]{ScTo}, the Orlik-Terao ideal of $\mathcal A_X:=\{H_1,\ldots,H_p\}$ has linear graded free resolution of length $p-2$. Since ${\rm IOT}(X)\subseteq {\rm IOT}(\C)$, and the minimal generators of the former ideal (which are quadratic) are also part of the minimal generating set of the latter, one obtains $$\delta\geq p-3.$$ Since $p=n-d$, one obtains (2).
\end{proof}

At this moment we are not aware as to generalizing Proposition \ref{result3.2} to arbitrary $k\geq 4$. For example, \cite[Theorem 3.4]{ScTo} says that the length of the linear strand of the Orlik-Terao ideal of a graphic arrangement is $\leq 1$ all the time, regardless of $k$.

\medskip

Equality in statement (2) is attained quite often. Two examples that violate this are the Braid arrangement and the non-Fano arrangement: in both cases $n-d=3$ and $\delta=1$. This is due to the fact that the linear dependence among the relations give an unexpected linear syzygy (in a way, a converse of \cite[Proposition 3.6]{ScTo}).

Suppose we have a relation on some 3-relations:

$$a_1r_1+\cdots +a_sr_s=0, a_i\neq 0$$ and suppose the support of the relation $r_l$ is $\Lambda_l=\{i_l,j_l,k_l\}$.

If $i\in \Lambda_1$ but $i\notin\cup_{j\neq 1}\Lambda_j$, then since the term $a_1y_i$ must be zero, we get $a_1=0$ which contradicts $a_1\neq 0$. So $$\forall l, \Lambda_l\subset \cup_{j\neq l}\Lambda_j,$$ or, in other words, every index occurs at least 2 times in two different supports of relations.

Assume $\cup_{j=1}^s\Lambda_j=\{1,\ldots,t\}$.

Suppose we take the union of all supports $\Lambda_i$ and account for the repeats. We are going to have $3s$ indices. Each index is between $\{1,\ldots,t\}$ and it occurs at least twice. Therefore $$3s\geq 2t.$$ This leads to the following lemma.

\begin{lem}\label{result3.3} Let $X_1,\ldots, X_s\in L_2(\mathcal A)$ and let $r_i\in F(\mathcal A_{X_i}), i=1,\ldots,s$ be 3-relations with supports $\Lambda_1,\ldots,\Lambda_s$, and with $\displaystyle|\cup_i\Lambda_i|=t$. If $3s<2t$, then there is no linear syzygy on $\partial(r_1),\ldots,\partial(r_s)$.
\end{lem}

Lemma \ref{result3.3} is saying that if one has few multiple points, then the length of the linear strand is determined solely by the maximum multiplicity of such a multiple point.

We end with a warning given by the following proposition, similar in flavor with \cite[Example 4.1]{ToVa}.

\begin{prop}\label{result3.4} The minimum distance of a linear code $\C$ is not determined solely by the graded betti numbers of the Orlik-Terao algebra of $\C$.
\end{prop}
\begin{proof} The following have been computed with \cite{GrSt}.

Consider the linear code $\C_1$ with parameters $[6,3,3]$ and $\Sigma_{\C_1}=(x,y,z,x-y,x-z,y-z)$. The minimal graded free resolution of ${\rm OT}(\C_1)$ is: $$0\longrightarrow S^2(-5)\longrightarrow S^3(-4)\oplus S^2(-3)\longrightarrow S^4(-2)\longrightarrow S\longrightarrow {\rm OT}(\C_1).$$

Consider the linear code $\C_2$ with parameters $[6,3,2]$ and $\Sigma_{\C_2}=(x,y,x+y,x-y,z,y-z)$. The minimal graded free resolution of ${\rm OT}(\C_2)$ is: $$0\longrightarrow S^2(-5)\longrightarrow S^3(-4)\oplus S^2(-3)\longrightarrow S^4(-2)\longrightarrow S\longrightarrow {\rm OT}(\C_2).$$

Same minimal graded free resolution, yet different minimum distances.
\end{proof}

\vskip .1in

\noindent{\bf Acknowledgement} We thank Tristram Bogart for comments and corrections on improving the readibility of our work.
%%%%%%%%%%%%%%%%%%%%%%%%%%%%%%%%%%%%%%%%%%%%%%%%%%%%%%%%
% Back to single space
\renewcommand{\baselinestretch}{1.0}
\small\normalsize % to get previous line to take
%%%%%%%%%%%%%%%%%%%%%%%%%%%%%%%%%%%%%%%%%%%%%%%%%%%%%%%%

\bibliographystyle{amsalpha}

\end{document}